\documentclass[12pt]{article}
\usepackage{amsmath}
\usepackage{epsfig}                     
\usepackage{amssymb}
\newtheorem{thm}{Theorem}
\newtheorem{cor}[thm]{Corollary}
\newtheorem{defn}{Definition}

\newtheorem{notation}{Notation}
\newtheorem{lemma}[thm]{Lemma}

\newtheorem{prop}[thm]{Proposition}

\newenvironment{proof} { \emph{Proof.} } { {\rule{2mm}{2mm}}\\}

\newcommand{\bea}{\begin{eqnarray*}}
\newcommand{\eea}{\end{eqnarray*}}

\newcommand{\R}{\mathbb{R}}
\newcommand{\Z}{\mathbb{Z}}
\newcommand{\Q}{\mathbb{Q}}

\newcommand{\bm}{\begin{pmatrix}}
\newcommand{\fm}{\end{pmatrix}}

\title{Quasisymmetric maps of boundaries of amenable hyperbolic groups}
\author{Tullia Dymarz \footnote{Partially supported by NSF grant 1207296.
Deptarment of Mathematics, University of Wisconsin, Madison, 480 Lincoln Dr. 53706.  
email: dymarz@math.wisc.edu } }
\begin{document}
\maketitle
\begin{abstract}
In this paper we show that if $Y=N \times \Q_m$ is a metric space where $N$ is a Carnot group endowed with the Carnot-Caratheodory metric then any quasisymmetric map of $Y$ is actually bilipschitz. The key observation is that $Y$ is the \emph{parabolic visual boundary} of a \emph{mixed type} locally compact amenable hyperbolic group. The same results also hold for a larger class of nilpotent Lie groups $N$.  
As part of the proof we also obtain partial quasi-isometric rigidity results for mixed type locally compact amenable hyperbolic groups. 
Finally we prove a rigidity result for uniform subgroups of bilipschitz maps of $Y$ in the case of $N= \R^n$.\\

%\noindent{\bf Mathematics Subject ClassiÞcation (2000).} 20F65, 30C65, 53C20.

%\noindent{\bf Keywords.} quasi-isometry, quasisymmetric map, negative curvature.
\end{abstract}
\section{Introduction}
Quasisymmetric maps were introduced in \cite{BA} as natural replacements for quasiconformal maps for metric spaces where classical quasiconformal maps do not make sense (see Section \ref{prelimSEC} for the definition). 
For many standard metric spaces (such as $\R^n$) quasisymmetric maps are much more abundant than bilipschitz maps.  For $Y=N \times \Q_m$ this is not the case.

\begin{thm}\label{mainthmsimple}  Let $N$ be a Carnot group with the Carnot-Caratheodory metric and $\Q_m$ the $m$-adics with the standard metric. Then any quasisymmetric map of  $Y=N \times \Q_m$  onto itself is bilipschitz.  
\end{thm}
Quasisymmetric maps are especially interesting since they are strongly linked to negative curvature geometry in that quasi-isometries of negatively curved spaces induce quasisymmetric maps of their visual boundaries. 
We are able to study $Y$ precisely because it is the parabolic visual boundary of a geodesic negatively curved space which we will denote $X_{N,\varphi,m}$. The space $X_{N,\varphi,m}$ is constructed out of a simplicial $m+1$ valent tree $T_{m+1}$ and a negatively curved homogeneous space $G_\varphi=N \rtimes_\varphi \R$ where $\R$ acts on $N$ by a one parameter group of dilations  $\varphi$. 
(For the precise definition of  $X_{N,\varphi,m}$ see Section \ref{millefeuilleSEC}).
Quasisymmetric maps of $Y=N \times \Q_m=\partial X_{N,\varphi,m}$ are induced by quasi-isometries of $X_{N,\varphi,m}$ while bilipschitz maps of $Y$ are induced by \emph{height-respecting} quasi-isometries (see Section \ref{catSEC} for the definition). Using coarse topology provided by \cite{FM3} we prove Theorem \ref{mainthm} by showing that all quasi-isometries of $X_{N,\varphi,m}$ are height-respecting. 

\begin{prop}\label{HRQIprop} All quasi-isometries of $X_{N,\varphi,m}$ are height-respecting.\end{prop}

If we allow $N$ to be any connected, simply connected nilpotent Lie group whose Lie algebra admits a derivation whose eigenvalues have positive real parts  then we arrive at the same conclusions using the same methods as above where now the Carnot-Caratheodory metric is replaced by a suitable metric $d_\varphi$. (For a description of the metric $d_\varphi$ see Section \ref{nchsSEC}).  Therefore Theorem \ref{mainthmsimple} will actually be proved as a special case of the following theorem. 

\begin{thm}\label{mainthm}  Let $N$ be a nilpotent Lie group with an admissible metric $d_\varphi$ and $\Q_m$ the $m$-adics with the standard metric. Then any quasisymmetric map of  $Y=N \times \Q_n$  onto itself is bilipschitz.  
\end{thm}

These results should be compared to results of Xie in \cite{X1,X2,X3,X4} which use the same dictionary but in the reverse direction to show that all quasi-isometries of certain negatively curved homogeneous spaces are height-respecting by proving directly that all quasisymmetric maps of their boundaries are bilipschitz.  

In fact, Theorem \ref{mainthm} can be thought of as an extension of Xie's results. 
In \cite{Hein}, Heintze showed that all negatively curved homogeneous spaces can be given as solvable Lie groups (see Section \ref{nchsSEC} for details). In a similar spirit \cite{CCMT}, Caprace et al. classify all locally compact amenable hyperbolic groups. 
They show that there are three types of non-elementary amenable hyperbolic locally compact groups: negatively curved homogeneous spaces, stabilizers of an end in the full automorphism group of a semi-regular locally finite tree, and combinations of the two via a warped product construction. It is this last type that we call \emph{mixed type}.
All mixed type amenable hyperbolic groups act properly on some $X_{N,\varphi,m}$. Our analysis provides partial quasi-isometric rigidity results for mixed-type non-elementary amenable hyperbolic locally compact groups. See also \cite{CORN}.

\begin{cor}\label{QIcor1} $X_{N,\varphi,m}$ is not quasi-isometric to $X_{N',\varphi',m'}$ if $m,m'$ are not powers of a common base or if $G_\varphi=N \rtimes_\varphi \R$ is not quasi-isometric to $G_{\varphi'}=N' \rtimes_{\varphi'} \R$. 
\end{cor}
A more detailed classification result can be deduced in many cases from 
 \cite{FM3,P2,Reiter} but we will not give the details here except in case of $N=\R^n$.
Note that if we specialize to $N= \R^n$ then $\varphi=\varphi_M$ is given by multiplication by a one parameter subgroup $M^t \in GL_n(\R)$ where $M$ is a matrix whose eigenvalues all have norm greater than one. 
Using work of \cite{FM3} we get a full quasi-isometry classification in this case.

\begin{cor}\label{QIcor2}  $X_{\R^n,\varphi_M,m}$ is quasi-isometric to $X_{\R^n, \varphi_{M'},m'}$ if and only if $m=r^i$, $m'=r^j$ for some $r,i,j \in \mathbb{N}$ and 
$M$ and $M'$ have absolute Jordan forms  that  are powers of each other in such a way so that 
$\ln \alpha_1 /\ln \alpha_1'= i-j$ where $\alpha_1, \alpha_1'$ are the smallest eigenvalues of $M$ and $M'$ respectively.
\end{cor}
We also extend results from \cite{D1} and prove that certain groups of bilipschitz maps of $\R^n \times \Q_m$  can be conjugated to be of a particularly nice form. We state the theorem here but all definitions are given in the appendix.

\begin{thm}\label{anothertukia} Let $U$ be a uniform separable subgroup of $Bilip_{{M}}(\R^n \times \Q_m)$ that acts cocompactly on the space of distinct pairs of points of $\R^n \times \Q_m$. Then 
$U$ can be conjugated  into $ASim_{{M}}(\R^n\times \Q_p)$ for some $p$.
\end{thm}
The proof follows Theorem 2 in  \cite{D1} very closely which is why we relegate this result to the appendix and provide only a brief outline.\\

\noindent{\bf Remark.}
Theorem \ref{mainthm} should hold in even more generality. For example it should hold for any space whose hyperbolic cone (see  \cite{BS}) satisfies the conditions imposed in Theorems 7.3 and 7.7 in \cite{FM3}. (See Section \ref{proofSEC} for the statements of these theorems.)

\subsection{Outline} Following some preliminaries in Section \ref{prelimSEC} we study the geometry and boundaries of $X_{N,\varphi,m}$ in Section \ref{catSEC}. We prove Theorem \ref{mainthm} and Proposition \ref{HRQIprop} along with Corollaries \ref{QIcor1} and \ref{QIcor2} in Section \ref{proofSEC}. In the appendix we prove Theorem \ref{anothertukia}.\\

\noindent{\bf Acknowledgements.} The author would like to thank Xiangdong Xie for useful conversations and  Yves de Cornulier, John Mackay, Xiangdong Xie and the referee for comments on an earlier draft. 

%\newpage
\section{Preliminaries}\label{prelimSEC}

\begin{defn}A map $f:X \to Y$ between metric spaces is called a
\begin{description}
\item[Quasisymmetric embedding] if for some homeomorphism $\eta:[0, \infty) \to [0, \infty)$ and any distinct $x,x',y \in X$
$$ \frac{d_Y(f(y),f(x))}{d_Y(f(y),f(x'))} \leq \eta\left( \frac{d_X(y,x)}{d_X(y,x')} \right).$$
\item[Bilipschitz embedding] if there exist positive constants $a,b>0$ such that for all $x,x'\in X$
$$a\ d(x,x') \leq d(f(x),f(x'))\leq b\ d(x,x').$$
If we can chose $a=1/K$ and $b=K$ then we say $f$ is $K$-\emph{bilipschitz}.
If we can chose $a=s/K$ and $b=sK$ then we say that $f$ is a ($K,s$)-\emph{quasi-similarity}. 
We say that a group of bilipschitz maps/quasi-similarities is \emph{uniform} if $K$ is uniform over all group elements. 
\item[(K,C)-Quasi-isometric embedding] 
if  there exist $K,C >0$ such that for all $x,x' \in X$
$$-C + \frac{1}{K} d_X(x,x') \leq d_Y(f(x),f(x')) \leq K d_X (x,x') +C$$
\item[S-Similarity] if there exists $S > 0$ such that for all $x,x' \in X$ 
$$ d_Y(f(x),f(x')) = S d_X (x,x')$$

\end{description}
\end{defn}

\section{$CAT(-k)$ spaces and their visual boundaries}\label{catSEC}

Let $X$ be a $CAT(-k)$ metric space. Riemannian manifolds with sectional curvature $\leq -k$ and their convex subsets are the prime examples of $CAT(-k)$ spaces. Gluing two $CAT(-k)$ spaces along closed convex subsets also results in a $CAT(-k)$ subspace. Note that up to rescaling the metric by a constant we can assume that $k=1$.

Given a $CAT(-1)$ space $X$ and $a \in \partial X$ and we define the \emph{Euclid-Cygan} metric on  
$\partial X - \{a\}$ as is done in the appendix in  \cite{HP}.
Let $\mathcal{H}$ be a horosphere centered at $a$ and $b,c \in \partial X - \{a\}$.
Let $b_t,c_t$ be two geodesics connecting $b,c$ to $a$ with $b_0,c_0 \in \mathcal{H}$. (Note that this orientation is the opposite of what is given in \cite{HP}). The Euclid-Cygan metric on $\partial X - \{a\}$ is given by 
$$d_{a,\mathcal{H}}(b,c)= \lim_{t \to -\infty} e^{\frac{1}{2}(2t+d_X(b_t,c_t))}.$$
We can also endow $\partial X - \{a\}$ with a \emph{visual parabolic} metric.
For any $b,c \in \partial X - \{a\}$ the visual parabolic metric is given by 
$$ \bar{d}_{a,\mathcal{H}}(b,c)=e^{t_0}$$ 
where $t_0$ is  the point at which $d_X(b_{t_0}, c_{t_0})=1$.
It is easy to see that the two metrics are bilipschitz equivalent since if $t\leq t_0$ then 
$$2(t_0-t)+1-C \leq d_X(b_t,c_t) \leq  2(t_0-t) + 1 + C $$
and so 
\bea
d_{a, \mathcal{H}}(b,c)  &=& \lim_{t \rightarrow -\infty} e^{\frac{1}{2}(2t + d_{X}(b_{t}, c_{t}))}\\
&\simeq& \lim_{t \rightarrow -\infty} e^{\frac{1}{2}(2t +2(t_0-t) +1)}\\
&\simeq& e^{t_0}=\bar{d}_{a, \mathcal{H}}(b,c).
\eea
 (See \cite{X1} for more details). \\
  
\noindent{\bf Remarks.} Note that the reason $e$ is chosen as a base here is because $X$ is $CAT(-1)$. If $X$ were $CAT(-k)$ the base would be $e^{\frac{1}{\sqrt{k}}}$ instead of $e$. We can in fact choose the base to be $e^\alpha$ for any $\alpha 
\leq 1$. This is equivalent to \emph{snowflaking} the metric. 

\begin{defn} For any $b \in \partial X - \{a\}$ there is a unique geodesic $b_t$ connecting $b$ to $a$ with $b_0\in \mathcal{H}$. We call such geodesic a 
\emph{vertical geodesic}. 
\end{defn}
In this paper we will only be working with $CAT(-k)$ spaces that have the property that for every $x \in X$ there is a vertical geodesic passing through $x$. 
In this case we can view $X$ as a union (not necessarily disjoint) of all vertical geodesics
$$X= \bigcup_{b\in \partial X- \{a\}} b_t.$$

\begin{defn} We endow $X$ with a \emph{height function} (this is just a horofunction) $h: X \to \R$
given by $h(x)=t$ if $x=b_t$. Note that if $x=b_t$ and $x={b'}_{t'}$ then necessarily $t=t'$.
\end{defn}
A quasi-isometry of $X$  induces a quasisymmetric map 
$$\partial f:\partial X -\{a\} \to \partial X -\{f(a)\} $$ with respect to either $d_{a, \mathcal{H}}$ or $\bar{d}_{a, \mathcal{H}}$. Since the two metrics are bilipschitz equivalent we will drop the distinction between them and simply write $d_{\infty}$.
\\

\noindent{}In the following sections we describe all of the $CAT(-1)$ spaces we will be working with. 
\subsection{Negatively curved homogeneous spaces}\label{nchsSEC}\label{boundarySEC}
Let $N$ be a connected, simply connected, nilpotent Lie group. We say that a one parameter group of automorphisms $\varphi_t$ is \emph{contracting} (resp. \emph{expanding}) if for all $g \in N$ we have $\varphi_t(g) \to1$  (resp. $\varphi_{-t}(g) \to 1$) as $t \to \infty$. 

If $\varphi_t:N \to N$ is a one parameter group of expanding automorphisms then
$G_\varphi = N \rtimes_\varphi \R$ is a negatively curved homogeneous space when endowed with an appropriate left invariant Riemannian metric \cite{Hein,CCMT}.
By \cite{Hein} we have that, up to rescaling the metric, all $N \rtimes_\varphi \R$ can be endowed with a $CAT(-1)$ metric. The geometry of these spaces have been studied by \cite{FM3,P2,Reiter}. Some of their boundaries have been analyzed in \cite{D1,DP,X1,X2,X3,X4}.

Since $G_\varphi$ is negatively curved, we can consider its parabolic visual boundary $\partial G_\varphi-\{\infty\}$ where $\infty$ is chosen so that 
vertical geodesics are given by 
$$\gamma_g(t)=(g,t)\in N \rtimes_\varphi \R.$$
We can identify $\partial G_\varphi-\{\infty\}$ with $N$
and then the height function is simply given by $h(g,t)=t$.

To get a better grasp of the metric $d$ on $G_\varphi$ we define a quasi-isometrically equivalent metric $\bar{d}$ on $G_\varphi$ that is easier to work with than a Riemannian metric. 
Let $d_N$ be a metric induced by a left invariant metric on $N$ and for each $t \in \R$ set  
$$d_t(g_1,g_2)= d_N(\varphi_{-t}(g_1), \varphi_{-t}(g_2))=\| \varphi_{-t}(g_1^{-1}g_2) \|.$$
Let $\bar{d}$ be the largest metric on $G_\varphi$ such that the vertical geodesics given above are actually geodesics 
and the distance on each level set $h^{-1}(t) =N \times \{t\} \simeq N$ is $d_t$. (See \cite{G} Section 1.3).
Note also that the level set $N\times \{t\}$ with the metric $d_t$ is exponentially distorted in $G_\varphi$. Using the metric $\bar{d}$ it is easy to verify that 
for two distinct vertical geodesics $\gamma_{g_0}$ and $\gamma_{g_1}$ we have 
$$\lim_{t \to  \infty} d(\gamma_{g_0}(-t),\gamma_{g_1}(-t)) = \infty \textrm{ and }
\lim_{t \to \infty} d(\gamma_{g_0}(t),\gamma_{g_1}(t)) = 0.$$
Note that if $d_{t_0}(g_0,g_1)=1$ then $1/K-C \leq d((g_0,t_0),(g_1,t_0))\leq K+C$
so that up to bilipschitz equivalence we can interpret the parabolic visual metric as 
$$d_\infty(g_0,g_1) = e^{t_0}$$
where $t_0$ is the smallest value at which $d_{t_0}(g_0,g_1)=1$. 
See Section 5 of \cite{X1} for more details.  

\subsubsection{Snowflaking}
It is worthwhile to note that by reparametrizing $\varphi_t$ as $\varphi'_t=\varphi_{\alpha t}$ we get boundary metrics on $G_\varphi$ and $G_{\varphi'}$ that are \emph{snowflake} equivalent. In particular $$d_{\infty,\varphi} = d_{\infty,\varphi'}^\alpha.$$
Note that there is always a range of admissible $\alpha$ that ensure that $d_{\infty,\varphi'}^\alpha$ is actually a metric. 
Alternatively we could have chose the base $a=e^\alpha$ instead of $e$ in our definition of boundary metric. 
This will be important later when we define the millefeuille space (see Section \ref{millefeuilleSEC}).

\subsection{Trees}
Any tree $T$ is a $CAT(-1)$ space so again by fixing a point $\xi$ at infinity we can define the parabolic visual boundary with respect to $\xi$. Picking a point at infinity induces an orientation on edges (towards the point at infinity). This in turn induces a height function: designate a base point vertex to be at height zero then use orientation to determine the heights of all of the other vertices. Vertical geodesics in $T$ are the geodesics that are compatible with the height function. The parabolic visual boundary is again just the set of vertical geodesics and in this case the Euclid-Cygan metric can be interpreted as the $e^{t_0}$ where $t_0$ is the height at which the two vertical geodesics first coincide. Note that for a tree 
we can define a parabolic visual metric $a^{t_0}$ for any base $a>1$ and still have it be a metric.

\subsection{Millefeuille space}\label{millefeuilleSEC}
Let $T_{m+1}$ be the regular $m+1$ valent tree with orientation such that each vertex has $m$ incoming edges and one outgoing edge. With this orientation there is a natural choice for $\infty \in \partial T_{m+1}$. Again, vertical geodesics are given by the coherently oriented infinite geodesics. In this case we can identify 
$\partial T_{m+1} - \{\infty\}$ with the $m$-adics $\Q_m$ (see \cite{FM1} for the identification). 

\begin{defn}[Millefeuille space] Let $G_\varphi=N \rtimes_\varphi \R$ be a negatively curved homogeneous space with height function $h_\varphi:G_{\varphi}\to \R$.  
Let $h_m: T_{m+1} \to \R$ be a height function.
The millefeuille space is defined to be
$$X_{N,\varphi,m} = \{ (x,y) \in G_\varphi \times T_{m+1} \mid h_\varphi(x)=h_m(y)\}$$ 
with the induced path metric rescaled by a factor of $1/\sqrt{2}$. 
\end{defn}
Alternatively we can view $X_{N,\varphi,m}$ as the metric fibration 
$$\pi:X_{N,\varphi,m} \to T_{m+1}$$
 where  $\pi^{-1}(\ell)$ is identified with  $G_\varphi$ for each coherently oriented line $\ell$ in $T_{m+1}$ via a height-preserving isometry. See section 7.2 \cite{FM3} for a precise definition of metric fibration.

This space $X_{N,\varphi,m}$ was first defined in Section 7 of \cite{CCMT}.  In that section, it is also noted that $X_{N,\varphi,m}$ is a $CAT(-k)$ space if $G_\varphi$ is $CAT(-k)$. This is because locally  $X_{N,\varphi,m}$ is obtained by glueing $m$ copies of $G_\varphi$, along closed convex subsets (namely along horoballs of $G_\varphi$).

\begin{defn}
Following the terminology coined by Farb-Mosher in \cite{FM3}, we call each $\pi^{-1}(\ell)$ a \emph{hyperplane} of $X$ and each $\pi^{-1}(v)$ a \emph{horizontal leaf}. 
\end{defn}
%
% BOUNDARY 
%

As with $G_\varphi$ and $T_{m+1}$ there is a natural choice of $\infty \in X_{N,\varphi,m}$
such that the vertical geodesics in $X_{N,\varphi,m}$ are precisely the geodesics that project to vertical geodesics in both $G_\varphi$ and $T_{m+1}$. There is also an obvious induced height function $h: X_{N,\varphi,m} \to \R$ given by $h(x,y)=h_\varphi(x)=h_m(y)$.

\begin{prop} \label{bmetricPROP}
For an appropriate rescaling of  $\varphi$, the parabolic visual boundary $\partial X_{N,\varphi,m} - \{\infty\}$ is bilipschitz equivalent to $(N \times \Q_m, d_{\varphi,m})$ where $d_{\varphi,m}$ is the maximum of the metrics $d_\varphi$ (on $N$) and $d_{\Q_m}$ (on $\Q_m$).  
\end{prop}
\begin{proof} 
Note that for any two distinct vertical geodesics $\gamma$ and $\gamma'$ we have three possible cases. If $\gamma$ and $\gamma'$ project to the same geodesic in $T_{m+1}$ then $\gamma,\gamma'$ both lie in the same hyperplane $f^{-1}(\ell)\simeq G_\varphi$ in which case we can identify $\gamma,\gamma'$ with vertical geodesics in $G_\varphi$  (namely $\gamma\simeq \gamma_{g_0}$ and $\gamma'\simeq \gamma_{g_1}$ for some $g_0, g_1 \in N$). Then 
$$d_{\infty, X}(\gamma, \gamma')=d_{\infty,G_\varphi}(g_0,g_1).$$
Likewise, if the projection of $\gamma$ and $\gamma'$ to $G_\varphi$ is the same then the two geodesics coincide above some $t_0 \in \Z$. In this case $\gamma$ and $\gamma'$ lie in the same copy of $T_{m+1}$ and so we have 
$$d_{\infty, X}(\gamma, \gamma')=d_{\infty, T_{m+1}}(\gamma,\gamma').$$
Finally the last case is when $\gamma$ and $\gamma'$ project to two different vertical geodesics in both factors. Nevertheless, eventually (above height $t_1$), these two geodesics lie in the same hyperplane $f^{-1}(\ell)\simeq G_\varphi$. Then, above $t_1$, we can identify $\gamma\simeq \gamma_{g_0}$ and $\gamma'\simeq \gamma_{g_1}$ for some $g_0, g_1 \in N$.  
If $d_{t_1}(\gamma(t_1),\gamma'(t_1))\geq 1$ then 
$$d_{\infty, X} (\gamma,\gamma')=d_{\infty,G_\varphi}(g_0,g_1)$$
since then the height at which $\gamma,\gamma'$ are distance one is above $t_1$.
Otherwise, $d_{t_1}(\gamma(t_1),\gamma'(t_1))< 1$ and the boundary metric has the property that 
$$1/K\ d_{\infty,T_{m+1}}(\gamma,\gamma') \leq d_{\infty, X}(\gamma, \gamma') \leq K\ d_{\infty,T_{m+1}}(\gamma,\gamma').$$
Finally, in order to get the standard metric on $\Q_m$ (i.e. $d_{\infty,T_{m+1}}(\gamma, \gamma')=m^{t_0}$ where $t_0$ is the height at which $\gamma,\gamma'$ initially come together) we must snowflake our boundary metric by $\alpha = \ln{m}$. To ensure that is possible we might have to replace $\varphi(t)$ with $\varphi'(t)=\varphi(\frac{1}{\alpha} t)$. (See the comments at the end of Section \ref{boundarySEC}).
\end{proof}

\subsection{Quasi-isometries}\label{QIofNCHS}

\begin{defn}\label{hrDEFN} Let $X$ be $CAT(-1)$ space with height function $h: X \to \R$. 
We say that a quasi-isometry $f:X \to X$ is \emph{height-respecting} if there is a constant $A$ such that $f$ maps any height level set of $h$ to within  Hausdorff distance $A$ of a height level set of $h$
and if the map induced on height is bounded distance from a translation.
In other words there exists a constant $a$ such that if $h(x)=t$ then 
$$-C + t+ a \leq   h(f(x))\leq t + a + C.$$
\end{defn}
It is now a well known fact (Lemma 6.1 \cite{X3} and Theorem 6.1 in \cite{FM1}) that for $G_\varphi$ and $T_{m+1}$  height-respecting quasi-isometries (up to bounded distance) are in one-to-one correspondence with bilipschitz maps of the parabolic visual boundary. Using the same arguments as in Lemma 6.1 \cite{X3} we can see that the same is true for $X_{N,\varphi,m}$.

\begin{prop}\label{hrtobilip} Any height-respecting quasi-isometry of $X_{N,\varphi,m}$ induces a bilipschitz map of the parabolic visual boundary
$$\partial X_{N,\varphi,m}-\{\infty\} \simeq N \times \Q_m.$$
\end{prop}
Roughly speaking the proof uses the fact that if for $\xi, \xi' \in \partial X_{N,\varphi,m}-\{\infty\}$  the vertical geodesics $\xi_t, \xi'_t$ are distance $\epsilon$ apart at height $t_0$ then the geodesics corresponding to their images are distance $\epsilon$ apart at height $t_1$ where $-C' + a + t_0 \leq t_1 \leq t_0 + a +C'$. 
Exponentiating this inequality gives the bilipschitz correspondence between distances. 
\\

\noindent{\bf Remark.}
Note that while we prove here that all quasi-isometries of $X_{N,\varphi,m}$ are height-respecting
it is not always the case that all quasi-isometries of $T_{m+1}$ and $G_\varphi$ are height-respecting. This is clear for $T_{m+1}$ however for $G_\varphi$ the answer is more subtle. 
For instance when $\varphi_t(x)=e^tx$ the space  $\R^n \ltimes_{\varphi} \R$ is isometric to $\mathbb{H}^{n+1}$ whose quasi-isometries can be identified with the quasiconformal maps of $S^n$. The other rank one symmetric spaces can also be written as $N \rtimes_\varphi \R$ for the appropriate $N$ and $\varphi$.
In \cite{X1,X2,X3}, Xie showed that when $\R^n \rtimes_\varphi \R$ is not isometric to $\mathbb{H}^{n+1}$ then all quasi-isometries are height-respecting.
In \cite{X4}, he was able to show that the same result is true for certain $N\rtimes_\varphi \R$.  It is an open question whether all quasi-isometries of negatively curved homogeneous spaces that are not isometric to symmetric spaces are height-respecting. \\

\section{Proof of Theorem \ref{mainthm}}\label{proofSEC}
In this section we prove Theorem \ref{mainthm} by proving Proposition \ref{HRQIprop} which states that any quasi-isometry of $X_{N,\varphi,m}$ is height-respecting.
We start with Farb-Mosher's Theorem 7.7 \cite{FM3}.

\begin{thm}[Theorem 7.7 in \cite{FM3}]\label{FMpreserves}
Let $\pi: X \to T$ be a metric fibration over a bushy tree $T$ such that the fibers of $\pi$ are contractible $k$-manifolds for some $k$. Let $f: X \to X$ be a quasi-isometry. Then there exists a constant $A$, depending only on the metric fibration data of $\pi$, the quasi-isometry data of $f$ and $T$ such that
\begin{enumerate}
\item For each hyperplane $P \subset X$ there exists a unique hyperplane $Q \subset X$ such that $d_{\mathcal{H}}(f(P),Q) \leq A$.
\item For each horizontal leaf $L \subset X$ there is a horizontal leaf $L'$ such that $d_{\mathcal{H}}(f(L),L') \leq A$.
\end{enumerate}
\end{thm}

We apply this theorem to $X=X_{N,\varphi,m}$ and $T=T_{m+1}$. 
Note that there are two possible types of hyperplanes in $X_{N,\varphi,m}$.
If $\ell$ is a coherently oriented geodesic in $T_{m+1}$ then $f^{-1}(\ell) \simeq G_\varphi$. Otherwise, $\ell$ changes orientation exactly once (say at $t_0$) so $f^{-1}(\ell) \simeq H$ is a doubled horoball complement (i.e. the union of two horoball complements $ \{ (g,t) \mid t \leq t_0\} \subset G_\varphi$ identified along the horosphere $N \times \{t_0\}$.)
Note also that if we endow $H$ with the induced metric from $X$ then $H$ is not a geodesic metric space. For instance, the points $p=(g_0, t_0)$ and $q=(g_1, t_0)$ are not connected by a geodesic. In fact, the induced path metric distorts the distances between $p$ and $q$ exponentially.  We will use this fact to prove the following proposition.
 
\begin{prop}\label{COHprop} A quasi-isometry $f:X_{N,\varphi,m}\to X_{N,\varphi,m}$ takes coherently oriented hyperplanes  to coherently oriented hyperplanes.
\end{prop}
\begin{proof} 
Let $P\subset X_{N,\varphi,m}$ be a coherently oriented hyperplane. 
Suppose $$f(P)\simeq Q$$ is a doubled horoball complement. 
Note that $f$ maps horizontal leaves in $P$ to within distance $A$ of horizontal leaves in $Q$. Let $p',q'$ be two points that are mapped to within distance $C$ of $p=(g_0,t_0)$ and $q=(g_1,t_0)$. Without loss of generality, we can assume that $p'=(g'_0,t'_0)$ and $q'=(g'_1,t'_0)$ for some $t_0'$.
Recall that by slightly enlarging the quasi-isometry constants we can assume that $f$ and its coarse inverse $\bar{f}$ are actually Lipschitz (see for example \cite{FM1}). This implies that $f$ (and $\bar{f}$) sends rectifiable paths of length $L$ to rectifiable paths of length at most $KL$.

Note that any path connecting $p$ and $q$ that lies in a $C$-neighborhood of $Q$ in $X$ has length at least $K'\exp(d(p,q))$ for some $K'>0$.
Note also that if $d_P$ denotes the distance in $P$ then the geodesic $\gamma$ between $p'$ and $q'$ has length 
$$d(p',q')=d_P(p',q')$$ 
since $P\subset X_{N,\varphi,m}$ is geodesically embedded. Now $len(\gamma)=d(p',q')$ and
$d(p',q') \leq Kd(p,q)+C''$ 
so 
$$len(f(\gamma)) \leq K len(\gamma)\leq K^2 d(p,q)+C'''.$$
But $f(\gamma)$ must lie in a $C$-neighborhood of $Q$ so 
$$len(f(\gamma)) \geq K' \exp(d(p,q)).$$
This implies that 
$$ K' \exp(d(p,q)) \leq K^2 d(p,q)+C'''$$ 
which is not possible if $d(p,q)$ is large enough. 
\end{proof}\\
Now we know that for each coherently oriented $P\subset X_{N,\varphi,m}$, the quasi-isometry $f$ induces a quasi-isometry of $G_\varphi$ that permutes height level sets.
From Proposition 5.8 in \cite{FM3} and Theorem 33 in \cite{Reiter} we have that any quasi-isometry of $G_\varphi$ that permutes height level sets is height-respecting where the constants  $A,a,C$ in Definition \ref{hrDEFN} depend only on  $G_\varphi$ and the quasi-isometry constants of $f$.
This in turn implies that $f$ induces a height-respecting quasi-isometry of $T_{m+1}$ and more generally of $X_{N,\varphi,m}$. 

To finish the proof of Theorem \ref{mainthm} we appeal to Proposition \ref{hrtobilip}. Since every quasi-isometry of $X_{N,\varphi,m}$ is height-respecting then any quasisymmetric map of $Y \simeq \partial X_{N,\varphi,m}$ is  bilipschitz.  \\

\noindent{\bf Remark.} Theorem \ref{FMpreserves}
(Theorem 7.7 from \cite{FM3}) relies on Theorem 7.3 from \cite{FM3} which we state below. As noted in the introduction. Theorem \ref{mainthm} should hold for the boundaries of any metric fibration of a tree with a negatively curved space satisfying Theorems 7.3 and 7.7.

\begin{thm}[Theorem 7.3 in \cite{FM3}]\label{FM7.3}
Let $\pi: X \to T$ be a metric fibration whose fibers are contractible $k$-manifolds for some $k$. Let $P$ be a contractible $(k+1)$-manifold which is a uniformly contractible, bounded-geometry, metric simplicial complex. Then for any uniformly proper embedding $\phi: P \to X$, there exists a unique hyperplane $Q \subset X$ such that $\phi(P)$ and $Q$ have finite Hausdorff distance in $X$. The bound on Hausdorff distance depends only on the metric fibration data for $\pi$, the uniform contractibility data and the bounded geometry data for $P$, and the uniform properness data for $\phi$.
\end{thm}

\subsection*{\bf Remark on the proof of  Corollaries \ref{QIcor1} and \ref{QIcor2}.}

Instead of only considering maps from a single millefeuille space $X_{N,\varphi, m}$ to itself 
we can easily adapt all definitions and theorems to cover the case of maps from $X_{N,\varphi, m}$ to another millefeuille space $X_{N',\varphi', m'}$. Indeed Theorem  \ref{FMpreserves} (Theorem 7.7 in \cite{FM3}) is originally stated in this this form. 
Proposition \ref{COHprop} also goes through unchanged and furthermore it gives us that any quasi-isometry $f: X_{N, \varphi,m} \to X_{N',\varphi',m'}$
induces a height-respecting quasi-isometry from $G_\varphi=N \rtimes_\varphi \R$ to $G_{\varphi'}=N' \rtimes_{\varphi'} \R$.  In the special case when $N,N'=\R^n$ and $\varphi, \varphi'$ are given by multiplication by one parameter subgroups $M^t, M'^{t} \in GL_n(\R)$ we know from \cite{FM3} that  height-respecting quasi-isometries  between $G_{\varphi}$ 
and $G_{\varphi'}$ exist if and only if the absolute Jordan forms of $M$ and $M'$ are powers of a common  matrix. 
Furthermore this height respecting quasi-isometry induces a map on height that is bounded distance from 
the affine map $t \mapsto ax+b$  where $a=\ln \alpha_1/\ln \alpha_1'$ and $\alpha_1, \alpha_1'$ are the smallest eigenvalues of $M$ and $M'$ respectively (Proposition 5.8 in \cite{FM3}). 

 If we modify the definition of height-respecting quasi-isometry to include maps 
from $X_{N,\varphi,m}$ to $X_{N',\varphi',m'}$ then  Proposition \ref{hrtobilip} also holds provided we chose the correct exponents for the boundary metrics.
This proposition combined with Lemma \ref{decompLEM} below shows us that any quasi-isometry between  
$X_{N,\varphi, m}$ and $X_{N',\varphi', m'}$
induces a bilipschitz equivalences between $\Q_m$ and $\Q_{m'}$. However, we know from the appendix of  \cite{FM1} that this is only possible if $m$ and $m'$ are powers of a common base. Finally, since the scaling must match up on both factors simultaneously we must have that $i-j=a= \ln \alpha_1/\ln \alpha_1'$.
These observations prove Corollaries \ref{QIcor1} and \ref{QIcor2}. 

%
% SECTION
%
%\newpage
\appendix

\section{Uniform subgroups of $Bilip_M(\R^n \times \Q_m)$}\label{unifSEC}
Let $N\simeq \R^n$ and $\varphi=\varphi_M$ where $\varphi_M(v)=Mv$ for some $M \in GL_n(\R)$ with all eigenvalues of norm greater than one. Without loss of generality we can assume that $M$ is in absolute Jordan form. Suppose $e^{\alpha_i}$ for $1 \leq i \leq r$ are the eigenvalues of $M$ with $\alpha_i < \alpha_{i+1}$. In case $M$ is diagonal, we give an explicit form for the visual metric on $\R^n \simeq \partial G_\varphi$. Let $v,w \in \R^n$ and $|\Delta x_i|$ the norm of the difference of $v$ and $w$ in the $e^{\alpha_i}$ eigenspace. Then a parabolic visual metric can be given by
$$D_M(v,w) :=d_\infty(v,w)= \max_i \{ |\Delta x_i|^\frac{\alpha_1}{\alpha_i}\}.$$
In case $M$ is not diagonalizable the metric $D_M$ is more cumbersome to write down so we refer the reader to \cite{DP} and \cite{X3}.

Note that by Proposition \ref{bmetricPROP}, a parabolic visual metric on the boundary of $X_{\R^n,\varphi,m}$ can be given by  
$$D_{M,m}((v,y),(v',y'))=\max \{ D_M(v,v'), d_{\Q_m}(y,y')\}$$
where $d_{\Q_m}$ is the standard metric on the $m$-adics. In this section we study subgroups of bilipschitz maps of $\R^n \times \Q_m$ with respect to the metric $D_{M,m}$.

%
% NOTATION
%
\begin{notation} Denote by
\begin{itemize}
\item $Bilip_M(\R^n \times \Q_m)$ the set of all bilipschitz maps of $\R^n \times \Q_m$ with respect to the metric $D_{M,m}$,
\item $Sim_M(\R^n \times \Q_m)$, the set of similarities with respect to $D_{M,m}$,
\item $ASim_{{M}}(\R^n\times \Q_m)$ 
the set of similarities composed with \emph{almost translations} (i.e. maps in $Bilip_M(\R^n \times \Q_m)$ of the form
$$(x_1,x_2, \cdots , x_r,y) \mapsto (x_1 + B_1(x_2, \cdots , x_r,y), \cdots , x_r + B_r(y), y).$$
Here the $B_i$ are $\frac{\alpha_i}{\alpha_j}$-H\"older in each $x_j$ and Lipschitz in $y$). 
\end{itemize}
\end{notation}
For our purposes we will call a subgroup of bilipschitz maps uniform if it is uniform as a group of quasi-similarities (see Section \ref{prelimSEC}).

We will now prove Theorem \ref{anothertukia} stated in the introduction.\\

\noindent{\bf Theorem \ref{anothertukia}}\emph{ Let $U$ be a uniform separable subgroup of $Bilip_{{M}}(\R^n \times \Q_m)$ that acts cocompactly on the space of distinct pairs of points of $\R^n \times \Q_m$. Then 
$U$ can be conjugated by a bilipschitz map into $ASim_{{M}}(\R^n\times \Q_p)$ for some $p$.
}\\

\noindent{}Before beginning the proof of this theorem it is worth mentioning that after some initial setup this proof follows very closely the proofs of Theorem  2 in \cite{D1} and Theorem 2 in \cite{DP}. For this reason we will provide only an outline of the proof and fill in the details of where the proofs differ. Furthermore when we do need to refer to the specifics of $D_M$, for simplicity we will only refer to the diagonal case.  What is new in this Theorem is the introduction of the $\Q_n$ coordinate. Note also that $p$ might be different from $m$.

\begin{lemma}\label{decompLEM} If $f \in Bilip_M(\R^n \times\Q_m)$ then $f$ decomposes as 
$$f(x,y)=(f_1(x,y), f_2(y))$$
where $f_1$ is  bilipschitz with respect to ${D_M}$ in the $x$ coordinate and Lipschitz in $y$ with respect to $d_{\Q_M}$. The map $f_2$ is bilipschitz in $y$ with respect to $d_{\Q_M}$.
\end{lemma}
\begin{proof} In this case the proof is simply topological. Since the connected components of $\R^n \times \Q_m$ are given by $\R^n \times \{y_0\}$ for each $y_0 \in \Q_m$ the map $f$   maps $\R^n \times \{y_0\}$ to $\R^n \times \{y_1\}$ and induces a bilipchitz map between these two sets. 
\end{proof}\\
Note that combining this with the structure of $Bilip_{D_M}$ maps from Proposition 4 in  \cite{D1} we see that 
$f$ has the form 
$$f(x_1, \ldots, x_r, y)=(f_1(x_1, \ldots, x_r,y), \ldots, f_r(x_r,y), g(y))$$
where $f_i$ are 
$\frac{\alpha_i}{\alpha_j}$-H\"older in each $x_j$ and Lipschitz in $y$.

In the proof of Theorem \ref{anothertukia} we combine Lemma \ref{decompLEM} along with the following theorem from \cite{MSW1}.
\begin{thm}[Theorem 7 in \cite{MSW1}]\label{MSWconj} Suppose that $U \subseteq Bilip(\mathbb{Q}_m)$ is a uniform subgroup. Suppose in addition that the induced action of $U$ on the space of distinct pairs in $\mathbb{Q}_m$  is cocompact. Then there exists $p \geq 2$ and a bilipschitz homeomorphism $\mathbb{Q}_m \mapsto \mathbb{Q}_p$ which conjugates $H$ into the similarity group $Sim(\mathbb{Q}_p)$. 
\end{thm}
\begin{proof}(Of Theorem \ref{anothertukia})   
By Lemma \ref{decompLEM} we have a natural projection $$\pi: Bilip_{{M}}(\R^n \times \Q_m) \to Bilip(\Q_m).$$
A uniform subgroup $U\subseteq Bilip_{{M}}(\R^n \times \Q_m ) $  projects to a uniform subgroup $\pi(U)$ of $Bilip(\Q_m)$. Since $U$ was assumed to act cocompactly on pairs of points in $\R^n \times \Q_m$ we have that $\pi(U)$ acts cocompactly on pairs of $\Q_m$. 
By Theorem 7 in \cite{MSW1} (Theorem \ref{MSWconj} above) we can conjugate $\pi(U)$ via  a bilipschitz map into $Sim(\Q_p)$ for some $p$ possibly different from $m$. 

After this conjugation we can assume that $U$ is a uniform subgroup of $Bilip_{{M}}(\R^m \times \Q_p)$ where $f \in U$ 
acts by 
$$f(x,y)=(f_1(x_1, \ldots, x_r,y), f_2(x_2,\ldots,x_r,y), \ldots,f_r(x_r,y), \sigma_f(y) )$$ 
and $\sigma_f$ is a similarity of $\Q_p$ and for each fixed $y$ the map 
$$f_y(x)=(f_1(x_1, \ldots, x_r,y), f_2(x_2,\ldots,x_r,y), \ldots,f_r(x_r,y))$$
is bilipschitz with respect to $D_M$.
We have further projections 
$$\pi_i: Bilip_M(\R^n \times \Q_p) \to Bilip_{M_i}(\R^{n_i } \times \Q_p)$$
given by $\pi_i(f)=(f_i, \ldots, f_r, \sigma_f)$. For each fixed $y\in \Q_p$ we have
$\pi_i(f)_y=(f_i, \ldots, f_r)$ is a bilipschitz map with respect to $D_{M_{i}}$ where $M_{i}$ is the $n_i \times n_i$ submatrix of $M$ that contains all eigenvalues greater than or equal to $\alpha_i$.

Now we can follow the conjugation proof found in Section 3 of \cite{D1}. The proof is easily adapted to our situation. The key change is that whenever a measure is called for we use  the product $\mu$ of Lebesgue measure $\lambda$ on $\R^n$ and Hausdorff measure $\nu$ on $\Q_p$. Following this proof, we will reprove any Lemmas that rely on measure theoretic arguments. 

The conjugation is done inductively on $i$ starting at $i=r$ and working down to $i=1$.
So for induction we assume that we have a uniform subgroup $U\subseteq Bilip_{M_i}(\R^{n_i} \times \Q_p)$ where each $f\in U$ has the form 
$$f(x_i, \ldots,x_r,y)=(f_i(x_i,\ldots,x_r,y), F({x}_{i+1},\ldots x_r,y), \sigma_f(y))$$
where $(F,\sigma_f) \in ASim_{M_{i+1}}(\R^{n_{i+1}}\times \Q_p)$. 
If $x_i$ represents a one dimensional subspace then Section 3.3 in \cite{D1} goes through unchanged.
If $x_i$ represents a higher dimensional subspace then we follow Section 3.4. 
First we find a $U$ invariant \emph{foliated conformal structure} i.e. a measurable assignment
$$\eta: \R^{n_i} \times \Q_p \to SL(n_i-n_{i+1},\R)/ SO(n_i-n_{i+1},\R)$$
such that $$\eta(x_i, \ldots, x_r, y)=f'_i(x_i,\ldots,x_r,y) [\eta( f(x_i,\ldots,x_r,y))]$$
where the derivative of $f_i$ is taken in the $x_i$ coordinate and the action is a standard matrix action on $SL((n_i-n_{i+1},\R)/ SO((n_i-n_{i+1},\R)$ (see Section 3.4 in \cite{D1} for more details).
Again this only uses differentiability in the $x_i$ direction and the group structure so no changes are needed in the proof. 

Next we follow Theorem 13 in \cite{D1} and prove that if $\eta$ is approximately continuous at some radial point then we can find a conjugating map that conjugates $f_i$ to be a similarity in the $x_i$ coordinate while preserving the fact that $F \in ASim_{M_{i+1}}(\R^{n_i} \times \Q_p)$.
Note that by Theorem 2.9.13 in \cite{Fed} we have that 
any measurable map from $\R^n \times \Q_p$ to a separable metric space is approximately continuous almost everywhere. In particular $\eta$ is approximately continuous almost everywhere. Furthermore since $U$ acts cocompactly on pairs of points of $\R^n\times \Q_p$ every point of $\R^n\times \Q_p$ is a radial point. 

The conjugating map is constructed from a sequence of $g_i \in U$ (the same sequence used in defining the radial point) rescaled by a family of similarities in $Sim_{M_i}(\R^{n_i} \times \Q_p)$ following \cite{D1}. The proof that the conjugation does what we claim it does (i.e. that $f_i$ is a similarity in $x_i$) requires a couple of lemmas on quasi-conformal maps. 
These are found in Section 3.4.2 in \cite{D1}. We restate and comment on one of these following this proof (Lemma \ref{qcLEM} below).

After the conjugation, each $f \in U$ has the form
$$f(x_i, \ldots,x_r,y)=( c_{x_{i+1},\ldots,x_r,y}A_{x_{i+1},\ldots,x_r,y}(x_i + B(x_{i+1},\ldots,x_r,y)), \quad \quad\quad \quad\quad \quad$$
$$\quad \quad\quad \quad\quad \quad\quad \quad\quad \quad\quad \quad\quad \quad\quad \quad\quad \quad\quad \quad F({x}_{i+1},\ldots x_r,y), \sigma_f(y))$$
and we must show that $c \in \R$ and $A \in O(n_i-n_{i+1},\R)$ do not depend on $x_{i+1},\ldots,x_r,y$. This follows sections 3.5 and 3.6 in \cite{D1}.
\end{proof}

\begin{lemma}[Lemma 10 in \cite{D1}]\label{qcLEM} Let $\mathcal{F}_{K,N}$ be a family of ($N,K$)-bilipschitz (quasi-similarity) maps in $Bilip_{M_i}(\R^{n_i} \times \Q_p)$  such that 
$$f(x_i,\cdots,x_r,y)=(f_i(x_i,\ldots, x_r,y), F(x_{i+1}, \ldots,y), \sigma_f(y))$$
where $f_i$ is bilipschitz in $x_i$ and $(F,\sigma_f)$ is in  $ASim_{M_{i+1}}(\R^{n_{i+1}}\times \Q_p)$. 
Then there exist $b,b'$ such that for any measurable set $E \subseteq \R^{n_i} \times \Q_p$ and all $f \in \mathcal{F}_K$
$$ b\mu(E)   \leq \mu(f(E)) \leq b'\mu(E).$$
\end{lemma}
\begin{proof}  Note that if $\sigma_f$ is an $N$-similarity then $\sigma_f$ distorts measure by a factor of $N$.
Otherwise this lemma is proved using Fubini's Theorem in the same way as the corresponding Lemma 10 in \cite{D1}.
\end{proof}

\bibliographystyle{amsalpha}

%\bibliography{QsymBilipBIB}

\begin{thebibliography}{CdCMT}

\bibitem[Ahl02]{Reiter}
Ashley~Reiter Ahlin, \emph{The large scale geometry of nilpotent-by-cyclic
  groups}, ProQuest LLC, Ann Arbor, MI, 2002, Thesis (Ph.D.)--The University of
  Chicago. \MR{2717040}

\bibitem[BA56]{BA}
A.~Beurling and L.~Ahlfors, \emph{The boundary correspondence under
  quasiconformal mappings}, Acta Math. \textbf{96} (1956), 125--142.
  \MR{0086869 (19,258c)}

\bibitem[BS00]{BS}
M.~Bonk and O.~Schramm, \emph{Embeddings of {G}romov hyperbolic spaces}, Geom.
  Funct. Anal. \textbf{10} (2000), no.~2, 266--306. \MR{1771428 (2001g:53077)}

\bibitem[dC]{CORN}
Yves de~Cornulier,
  \emph{On the quasi-isometric classification of focal hyperbolic groups}, preprint.

\bibitem[CdCMT]{CCMT}
Pierre-Emmanuel Caprace, Yves de~Cornulier, Nicolas Monod, and Romain Tessera,
  \emph{Amenable hyperbolic groups}, preprint.

\bibitem[DP11]{DP}
Tullia Dymarz and Irine Peng, \emph{Bilipschitz maps of boundaries of certain
  negatively curved homogeneous spaces}, Geom. Dedicata \textbf{152} (2011),
  129--145. \MR{2795238}

\bibitem[Dym10]{D1}
Tullia Dymarz, \emph{Large scale geometry of certain solvable groups}, Geom.
  Funct. Anal. \textbf{19} (2010), no.~6, 1650--1687. \MR{2594617
  (2011c:20075)}

\bibitem[Fed69]{Fed}
Herbert Federer, \emph{Geometric measure theory}, Die Grundlehren der
  mathematischen Wissenschaften, Band 153, Springer-Verlag New York Inc., New
  York, 1969. \MR{0257325 (41 \#1976)}

\bibitem[FM98]{FM1}
Benson Farb and Lee Mosher, \emph{A rigidity theorem for the solvable
  {B}aumslag-{S}olitar groups}, Invent. Math. \textbf{131} (1998), no.~2,
  419--451, With an appendix by Daryl Cooper. \MR{1608595 (99b:57003)}

\bibitem[FM00]{FM3}
\bysame, \emph{On the asymptotic geometry of abelian-by-cyclic groups}, Acta
  Math. \textbf{184} (2000), no.~2, 145--202. \MR{1768110 (2001e:20035)}

\bibitem[Gro87]{G}
M.~Gromov, \emph{Hyperbolic groups}, Essays in group theory, Math. Sci. Res.
  Inst. Publ., vol.~8, Springer, New York, 1987, pp.~75--263. \MR{919829
  (89e:20070)}

\bibitem[Hei74]{Hein}
Ernst Heintze, \emph{On homogeneous manifolds of negative curvature}, Math.
  Ann. \textbf{211} (1974), 23--34. \MR{0353210 (50 \#5695)}

\bibitem[HP97]{HP}
Sa'ar Hersonsky and Fr{\'e}d{\'e}ric Paulin, \emph{On the rigidity of discrete
  isometry groups of negatively curved spaces}, Comment. Math. Helv.
  \textbf{72} (1997), no.~3, 349--388. \MR{1476054 (98h:58105)}

\bibitem[MSW03]{MSW1}
Lee Mosher, Michah Sageev, and Kevin Whyte, \emph{Quasi-actions on trees. {I}.
  {B}ounded valence}, Ann. of Math. (2) \textbf{158} (2003), no.~1, 115--164.
  \MR{1998479 (2004h:20055)}

\bibitem[Pan89]{Pansu}
Pierre Pansu, \emph{M\'etriques de {C}arnot-{C}arath\'eodory et
  quasiisom\'etries des espaces sym\'etriques de rang un}, Ann. of Math. (2)
  \textbf{129} (1989), no.~1, 1--60. \MR{979599 (90e:53058)}

\bibitem[Pen11]{P2}
Irine Peng, \emph{Large scale geometry of nilpotent-by-cyclic groups}, Geom.
  Funct. Anal. \textbf{21} (2011), no.~4, 951--1000. \MR{2827016}

\bibitem[SX]{X1}
N~Shanmugalingam and Xiangdong Xie, \emph{A rigidity property of some
  negatively curved solvable lie groups}, to appear in Comment. Math. Helv.

\bibitem[Xiea]{X3}
Xiangdong Xie, \emph{Large scale geometry of negatively curved
  $\mathbb{R}^n\rtimes \mathbb{R}$}, preprint.

\bibitem[Xieb]{X4}
\bysame, \emph{Quasisymmetric maps on reducible carnot groups}, preprint.

\bibitem[Xiec]{X2}
\bysame, \emph{Quasisymmetric maps on the ideal boundary of a negatively curved
  solvable lie group}, to appear in Mathematische Annalen.

\end{thebibliography}

\def\cprime{$'$} \def\cprime{$'$}
\providecommand{\bysame}{\leavevmode\hbox to3em{\hrulefill}\thinspace}
\providecommand{\MR}{\relax\ifhmode\unskip\space\fi MR }
\providecommand{\MRhref}[2]{%
  \href{http://www.ams.org/mathscinet-getitem?mr=#1}{#2}
}
\providecommand{\href}[2]{#2}

\end{document}